\documentclass[11pt]{amsart}
\usepackage{amsfonts,amsmath,amssymb,amsthm}
\usepackage{enumerate}
\usepackage{hyperref}
\usepackage{interval}
\intervalconfig{soft open fences}

\usepackage{amssymb}

\theoremstyle{plain} 
\newtheorem{theorem}{Theorem}
\newtheorem{corollary}[theorem]{Corollary}
\newtheorem{lemma}[theorem]{Lemma}
\newtheorem{proposition}[theorem]{Proposition}
\theoremstyle{definition}

\newtheorem{definition}{Definition}

\theoremstyle{remark}
\newtheorem{remark}{Remark}

\newcommand{\ett}{\mathsf{1}}

\newcommand{\sgn}{{\mbox{sgn}}}
\newcommand{\qr}[1]{\eqref{#1}}

\newcommand{\RR}{{\mathbb{R}}}
\newcommand{\ZZ}{{\mathbb{Z}}}
\newcommand{\NN}{{\mathbb{N}}}

\newcommand{\CB}{\mathcal{B}}
\newcommand{\CF}{\mathcal{F}}
\newcommand{\CM}{\mathcal{M}}

\def\C{\mathcal{C}}
\def\CK{\mathcal{K}}

\def\F{\mathit{F}}
\def\M{\mathit{M}}

\def\E{\mathsf{E}}
\def\P{\mathsf{P}}
\def\T{\mathit{T}}
\def\L{\mathit{L}}
\def\Ls{{\L^{\!*}}}
\newcommand{\var}{\operatorname{var}}
\let\tl=\tilde

\let\X=X
\let\x=x

\newcommand{\ii}[1]{^{(#1)}}
\def\Ordo#1{O\left(#1\right)}
\def\dtv{d_{\mathrm{TV}}}

\begin{document}

\title{Doeblin measures\\ uniqueness and mixing properties} \author{Noam Berger,
  Diana Conache, Anders Johansson, and Anders \"Oberg}

\address{Noam Berger and Diana Conache: Technische Universit\"at M\"unchen,
  School of Computation, Information and Technology, Boltzmannstrasse 3 85748
  Garching bei M\"unchen, Germany.}
  \email{noam.berger@tum.de}
  \email{diana.conache@tum.de}
  \address{Anders Johansson, Department of Mathematics,
  University of G\"avle, 801 76 G\"avle, Sweden.}
  \email{ajj@hig.se}
\address{Anders \"Oberg, Department of Mathematics, Uppsala University, P.O.\
  Box 480, 751 06 Uppsala, Sweden.}
  \email{anders@math.uu.se}

\date{} 

\begin{abstract}
  In this paper we solve two open problems in ergodic theory. We prove first
  that if a Doeblin function $g$ (a $g$-function) satisfies
  \[\limsup_{n\to\infty}\frac{\var_n \log g}{n^{-1/2}} < 2,\]
  then we have a unique Doeblin measure ($g$-measure). This result indicates
  a possible phase transition in analogy with the long-range Ising model.

  Secondly, we provide an example of a Doeblin function with a unique Doeblin
  measure that is not weakly mixing, which implies that the sequence of iterates
  of the transfer operator does not converge, solving a well-known folklore problem in ergodic theory.
  Previously it was only known that uniqueness does not imply the Bernoulli property.
\end{abstract}

\maketitle

\section{Introduction}\label{s.intro}\noindent

In this paper we provide, in terms of the variations of a $g$-function, the so
far best known (and possibly sharp) uniqueness condition for what is usually
called a $g$-measure. Keane introduced the terminology of $g$-functions and
$g$-measures in a groundbreaking paper (1972,~\cite{keane}) and this terminology
became standard after further important papers by Ledrappier (1974,~\cite{led})
and Walters (1975,~\cite{walters1}). These measures had already appeared in the
pioneering probability theory work by Doeblin and Fortet (1937,~\cite{doeblin}),
where they proved uniqueness of stationary distributions for what they called 
chains of complete connections. This is possibly the first time coupling 
was used in probability theory.

We propose to change the name of $g$-measures to Doeblin measures, and also $g$-functions to
Doeblin functions. We think the name $g$-measure is rather confusing, since it just refers to some function
$g$. Also, many authors have used other names than $g$-measures, such as (stationary
distributions of) stochastic chains with unbounded memory (\cite{gallesco}).

Moreover, we would also like to pay tribute to Wolfgang Doeblin, a French
Mathematician of German-Jewish origin, who died early in the Second World War,
leaving some great contributions to probability theory. We refer to Lindvall's
biographical paper~\cite{lin2} for more on Doeblin's life and work. 

In Section 2, we show that we have a unique Doeblin measure if the Doeblin function
$g>0$ satisfies
\begin{equation}\label{eq:assumption}
  \limsup_{n\to\infty}\frac{\var_n \log g}{{n}^{-1/2}} < 2.
\end{equation}
We were
motivated by the earlier result by Johansson, \"Oberg and
Pollicott~\cite{johob3}, where they proved that uniqueness follows if
\begin{equation*}
  \var_n \log g=o(n^{-1/2}),
\end{equation*}
as $n\to \infty$.

We conjecture that the condition~\eqref{eq:assumption} is sharp, i.e.\ that there are
systems where a phase transition occurs in analogy with the Ising models with
long-range interactions (see e.g.~\cite{aizenman}). The best known bound in this
direction is that provided in Berger, Hoffman and Sidoravicius~\cite{berger}
showing that the power $-1/2$ in the denominator of~\eqref{eq:assumption} is sharp. To
show that there exists a Doeblin function $g$ possessing multiple Doeblin
measures for which the left hand side of~\eqref{eq:assumption} is finite remains
however an open challenge.

In Section 3 we treat the well-known problem whether a unique Doeblin measure is
mixing, in the sense that uniqueness of a Doeblin measures implies that the
sequence of the iterates of the transfer operator converges. We answer this
negatively by means of a counterexample, and this also provides an example of a
unique Doeblin measure for which even weak mixing fails. Kalikow~\cite{kalikow}
gave an example of a system that is $K$ (implying uniqueness of a Doeblin
measure) but not Bernoulli, that is, that the system is (not) measurably
isomorphic to a Bernoulli shift.

For our counterexample in Section 3 we use the result of \cite{berger}, 
which gives examples of functions $g>0$ possessing
multiple Doeblin measures. We have supplied this paper with an appendix in order to clarify that
the construction in~\cite{berger} gives \emph{precisely} two extremal (ergodic)
measures, since this is important for our construction. This is not known for the construction of
more than one Doeblin measure in the paper by Bramson and Kalikow \cite{bramson}.

\noindent {\bf Acknowledgements}. We would like to thank Mark Pollicott and Peter Walters for valuable discussions. 
The fourth author wishes to thank the Knut and Alice Wallenberg Foundation for financial support.
\newpage
\subsection{Preliminaries}\label{sec:prel}

\subsubsection{Notation and basic definitions}
Let $(X,T)$ be a topological space and let $\CB$ be the corresponding Borel $\sigma$-algebra. 
We let $\C(X)$ denote the space of continuous
functions on $\X$ and $\CM(X)$ denote the space of probability measures on $X$.
We write $\delta_x\in\CM$ for the Dirac measure at $x\in\X$. Assuming the variable $x$
for values in $X$ and given a measure $\mu\in\CM(X)$, we write $x\sim\mu$ to (somewhat
informally) state that a variable $x$ has (or is sampled with) distribution $\mu$.
The expression $\mu(f(x))$ refers to the distribution $\mu \circ f^{-1}$ of the function
$f(x)$ defined on $X$. Similarly, the expression $\mu( f(x) | g(x))$ signifies the
\emph{conditional distribution} of $f(x)$ with respect to the sigma-algebra
generated by $g$.

We always assume the existence of an underlying probability space $\Omega$ with probability measure
$A\mapsto \P(A)$ and expectation operator $f\mapsto \E(f)$. We write discrete time
stochastic processes defined on $\Omega$ as $u\ii n$, i.e.\ with the time index
$n\in\ZZ$ as a superscript in parentheses.

\subsubsection{The state space $\X$}
Let $A$ be a finite set and let $\X=A^\NN$ be the space of one-way infinite
strings written left-to right as $x=x_0x_1\dots$. We endow $\X$ with the product
topology and the corresponding Borel sigma-algebra $\CB$. We consider the
symbolic shift system $(X,\T)$ where $\T:\X\to\X$, $x_0x_1\dots \mapsto \T\x = x_1x_2\cdots$,
is the destructive left shift on $\X$. The fiber at $x$ is thus the set
\(\T^{-1}x=\{ax:a\in A\}\), where $ax$ denotes the string $x$ with prefix $a\in A$
prepended.

We will occasionally use the following \emph{natural extension} $(\bar X,\bar T)$ of
$(\X,\T)$ with projection $\pi:\bar\X\to\X$ such that $\pi\circ \bar\T = \T\circ\pi$. We take
$\bar\X=A^\ZZ$ with the projection $\pi:\bar\X\to\X$ given by
\[
  \bar x= {(\bar\x_i)}_{i\in\ZZ} \mapsto \pi(\bar x) = \bar x_0 \bar x_{-1} \bar x_{-2} \cdots
\]
and the non-destructive right shift $\bar\T:\bar\X\to\bar\X$ defined by
${(\bar\T\bar x)}_i=\bar x_{i-1}$.

For $x\in\X$, let $\iota_n(x)\in A^n$ denote the initial prefix of $x$ of length $n$,
i.e.\ \(\iota_n(x)=x_0x_1\dots\x_n\) and let \({[x]}_n = \iota_n^{-1}(\iota_n(x))\) be the
initial cylinder set of length $n$. Denote by $\CF_n$ the sigma-algebra
generated by the initial cylinders. For $x,y \in \X$, we denote by $\kappa(x,y)\ge0$ the
length of initial \emph{agreement}, i.e.\
\begin{equation}\label{defkappa}
  \kappa(x,y) = \sup \{n\ge0 : \iota_{n}(x) = \iota_{n}(y)\}.
\end{equation}
The \emph{variation of order $n$}, $\var_n f$, $n\ge0$, of a function $f:\X\to\RR$
is
\[\var_n f = \sup_{x,y}\{|f(x)-f(y)| : \kappa(x,y) \ge n\}. \]
A function $f(x)$ is \emph{continuous} with respect to the product topology
precisely if $\var_n f\to0$ as $n\to\infty$. 

\subsubsection{Doeblin chains and Doeblin measures}
A chain with complete (or infinite) connections is a generalization of Markov
chains to having a continuous dependence on an infinite past, instead of a
finite past. Our study focuses on discrete time stochastic processes
$x\ii{n}\in\X$ with state space $\X$, such that, for all times $n$, the previous
state $x\ii{n-1}$ is the shift $\T x\ii n$ of the current state $\x\ii{n}$.
Hence, the process is deterministic in the direction of the past and has a
stochastic \emph{transition} in the direction of the future: Schematically,
\begin{equation}\label{Umarkov}
  \x\ii{n-1} \overset{\ \T}\longleftarrow x\ii{n} \; \rightsquigarrow\;  x\ii{n+1}
\end{equation}
where $x\ii {n+1} = a x\ii {n}$ is obtained by prepending a random prefix $a\in A$
to $x\ii n$. The process $x\ii n$ trivially has the Markov property since the
current state determines all past states: We can capture the full realisation
$\{x\ii n\}$ from an element $\bar x\in \bar\X$ by
\[
  x\ii{n}(\bar\x) = \pi(\bar\T^{-n}\,\bar\x) = \bar x_n \bar x_{n-1} \cdots.
\]
At each transition $\x\ii n\rightsquigarrow x\ii{n+1}$, we thus \emph{expose} the
coordinate $\bar\x_{n+1}$ of $\bar x$.

We also require that, given the current state $x\ii n$, the transition should sample
the next state $x\ii{n+1}$ in a time invariant manner and have a \emph{continuous}
dependence on the past. This means that
\begin{equation}\label{g-transition}
  \P\left(x\ii{n+1}=ax \mid\x\ii {n} =\x\right) = g(ax)
\end{equation}
where $g: \X \to [0,1]$ is a continuous function such that
\begin{equation}\label{doeblin-function}
  \sum_{a\in A} g(ax) = 1\quad\text{ for all }\x\in\X.
\end{equation}
\begin{definition}
We say that $g\in\C(\X)$ satisfying~\eqref{doeblin-function} is a \emph{Doeblin
  function} (a.k.a.\ a $g$-function) and if $g>0$ we say that $g$ is
\emph{regular}. For a given Doeblin function $g(x)$, we refer to processes
$\{ x\ii{n} \}$ satisfying~\eqref{g-transition} as \emph{$g$-chains}. 
\end{definition}
Note
that~\eqref{Umarkov} is a consequence of~\eqref{g-transition}. Without reference
to a fixed Doeblin function $g$, we speak of \emph{Doeblin chains}. We can
specify a Doeblin chain by specifying the \emph{system} $(\X,\T,g)$ together
with an initial distribution $\nu\in\CM(\X)$ for the state $x\ii0$ at time zero.

An alternative way to describe $g$-chains uses the associated \emph{transfer
  operator}, $\L=\L_g$, which acts on continuous functions $f\in \C(\X)$ according
to
\[
  f(x) \mapsto \L f(x) = \sum_{y\in\T^{-1}(x)} g(y) f(y).
\]
Then $\L$ is a transition operator for the process $\x\ii n$, i.e.
\(\L f(x) = \E\left(f(x\ii {n+1}) \mid\x\ii{n}=x\right)\) for all $n$ and all
functions $f\in\C(\X)$. For the dual operator ${\Ls}$ acting on $\CM(\X)$, it holds
that the conditional distribution
\(\P(x\ii{n+1}\mid\x\ii{n}) = \Ls \, \delta_{x\ii{n}}\).

A measure $\mu\in\CM(\X)$ is the distribution of the $n$:th term in a $g$-chain
governed by~\eqref{g-transition} if and only if $\mu=\Ls^n(\mu\circ\T^{-n})$. We refer
to such $\mu$ as $g$-chains of \emph{age $n$}. Denote by $\CK_g\ii{n} \subset \CM$ the
set of $g$-chains of age $n$ and $\CK_g = \cap_{n=1}^\infty \CK_g\ii{n}$ the space
$\CK_g\subset\CM(\X)$ of $g$-chains of infinite age.

\begin{definition}
A \emph{Doeblin measure} is the invariant distribution of a \emph{stationary}
$g$-chain. 
\end{definition}
Equivalently, a Doeblin measure $\mu\in\CM(\X)$ is a fixed point to the
dual transfer operator $\Ls$, i.e.\ \(\mu=\Ls \mu\). We denote by $\CM_g$ the space
of Doeblin measures corresponding to $g$. A Doeblin measure $\mu$ is translation
invariant, i.e. $\CM_g \subset \CM_\T$, where $\CM_\T:=\{\mu: \mu=\mu\circ\T^{-1}\}$. Note
that a translation invariant measure $\mu$ is a $g$-chain of infinite age and that
$\CM_g = \CM_\T \cap \CK_g\ii 1$.

Our first result, Theorem~\ref{thm:main}, concerns regularity conditions on $g$
ensuring that $\CM_g = \{\mu\}$, i.e.\ conditions implying the uniqueness of the
Doeblin measure. In our second result, Theorem~\ref{thm:counter}, we construct
an example of a system $(\X,\T,\tl g)$, where we have multiple Doeblin chains
but a unique Doeblin measure, i.e.\ $\CK_{\tl g} \supsetneq \CM_{\tl g} = \{\mu\}$.

\subsection{Previous results on uniqueness and non-uniqueness}

We say that $g(x)$ is \emph{local} if $g$ is $\CF_N$-measurable for some $N$,
i.e.\ determined by a finite set of coordinates. For a local Doeblin function
$g(x)$ the corresponding $g$-chains $x\ii{n}$ reduce to finite-state Markov
chains $\iota_N(x\ii{n})$ with state space $A^N$ and $g$ is therefore uniquely ergodic
if we furthermore assume that it is \emph{regular}, i.e.\ that $g > 0$.

For a system $(\X,\T,g)$, the existence of a Doeblin measure follows from the
Schauder-Tychonoff fixed-point theorem, since the compact convex set $\CM_T$ of
translation invariant measures maps to itself under the action of $\Ls$. Note
that the requirement that the Doeblin function $g$ is continuous is essential.
Consider for example the ``discontinuous Doeblin function'' $g(x)$,
$x\in{\{-1,1\}}^\NN$, given by
\[
  g(x) = 0.5 - 0.25\cdot\x_0 \cdot {\bf 1}_{M(x) > 0},
\]
where $M(x)=\limsup_{n\to\infty}\frac 1n \sum_{k=1}^n\x_{k}$. For this $g$ the space
$\CM_g$ is empty.

Doeblin and Fortet~\cite{doeblin} proved that the Doeblin measure is unique
whenever the regular Doeblin function $g$ satisfies ``summability of
variations''
\[
  \sum_{n=1}^\infty \var_n g<\infty.
\]
In~\cite{johob}, Johansson and Öberg showed that the condition of ``square
summability'' (or $\ell^2$-summability)
\begin{equation}\label{square}
  \sum_{n=1}^{\infty} {(\var_n g)}^2<\infty,
\end{equation}
implies uniqueness. In~\cite{johob2} Johansson, \"Oberg and Pollicott
extended this result to countable state shifts.

The condition of square summable variations was further explored by Gallesco,
Gallo and Takahashi~\cite{gallesco}. They also made a distinction between
equilibrium and dynamic uniqueness (similar to the distinction between weak and
strong uniqueness of Gibbs measures of disordered spin systems in rigorous
statistical mechanics~\cite{vanenter1}). Equilibrium uniqueness is what we call
uniqueness of a Doeblin measure, and dynamic uniqueness is a stronger notion.
They show that a.e.\ square summable variations (first introduced
in~\cite{johob}) is equivalent to dynamic uniqueness and also to the weak
Bernoulli property.

Berbee's results from the late 1980s~\cite{berbee, berbee2} are also intriguing
in this context. He proved that there is a unique Doeblin measure if
\begin{equation}\label{berbee}
  \sum_{n=1}^\infty e^{-r_1-\cdots-r_n}=\infty,
\end{equation}
where $r_n=\var_n \log g$ denotes the (logarithmic) variations of the Doeblin
function. This condition for uniqueness is not comparable with that of
\eqref{square}, i.e., there are cases where one is valid but not the other,
although it is worth noting that Berbee's condition is valid for $r_n=1/n$, but
not for $r_n=1/n^{\alpha}$, if $\alpha<1$, whereas~\eqref{square} is valid if $\alpha> 1/2$.
Also, Berbee's condition is not valid if $r_n=c/n$ and if $c>1$.

In~\cite{johob3}, the authors found a condition for uniqueness that subsumes
Berbee's condition~\eqref{berbee} for uniqueness with square summable
variations~\eqref{square}. As a corollary, uniqueness was obtained under the
condition
\[\var_n \log g=o\left(\frac{1}{\sqrt{n}}\right), \quad n\to \infty.\]

Fernandez and Maillard~\cite{fernandez} proved that uniqueness and mixing
properties for a Doeblin measure (and some generalisations) follow if the
Doeblin function satisfies the classical Dobrushin condition, which is
incomparable with variational conditions. In statistical mechanics, Bissacot,
Endo, Le Ny and van Enter~\cite{vanenter3} show that the ``$g$-measure
property'' is violated for certain Gibbs measures in low temperatures. Other papers that draw inspiration
from statistical mecahanics are Verbitskiy~\cite{verb}, Johansson, Öberg and
Pollicott~\cite{johob4}, and the forthcoming~\cite{johob5}. In the last two
papers the difference between conditions for uniqueness of Gibbs measures for
general potentials and the uniqueness of Doeblin measures are further explored
and to some extent clarified.

Concerning non-uniqueness of Doeblin measures, we refer to the breakthrough of
Bramson and Kalikow~\cite{bramson}, who provided the first counterexample to
uniqueness of Doeblin measures consistent with a positive Doeblin function,
although the Doeblin function is continuous. The paper by Berger, Hoffman and
Sidoravicius~\cite{berger} was a direct response to~\cite{johob} and provided
non-unique systems where the variations of the Doeblin functions are in
$\ell^{2+\epsilon}$ for $\epsilon>0$, proving that the square summable variations condition is
sharp for uniqueness in this sense. We would also like to mention the paper of
Hulse~\cite{hulse}, where he demonstrates, using an example, that the Dobrushin
type condition in~\cite{fernandez} cannot be essentially improved.

\subsection{Our results}
We now present the main two results of this paper.
\def\bct{c_1}

\begin{theorem}\label{thm:main}
  Let $g$ be a regular Doeblin function that satisfies \eqref{eq:assumption}, namely
  \begin{equation*}
    \limsup_{n\to\infty}\frac{r_n}{\sqrt{n}^{-1}} < 2, \text{ with } r_n:=\var_n \log g.
  \end{equation*}
  Then $g$ possesses a unique Doeblin measure.
\end{theorem}

\begin{theorem}\label{thm:main2}
There exists a Doeblin function with a unique Doeblin measure which is not weakly mixing.
\end{theorem}
This implies that the sequence of the iterates of the transfer operator does not converge.

\begin{remark} It is natural to ask whether there exists a positive constant
  $c^\ast$ such that if
\[
    \var_n \log g \le \frac {c^\ast}{\sqrt n}, \text{ for all } n \text{ large, }
\]
then uniqueness holds, but for all larger constants one can find an example of
non-uniqueness, i.e.\ whether a phase transition occurs when the variations of
$\log g$ are of order $1/\sqrt{n}$. We conjecture that this is indeed the case.
In Theorem~\ref{thm:main}, we provide 2 as a lower bound for $c^\ast$. If a phase
transition is established, we might ask whether 2 is indeed the critical value.
\end{remark}

\begin{remark}
  Sometimes authors use the variations of $g$, $\var_{n} g$ rather than
  $\var_{n} \log g$, in their conditions for uniqueness (e.g.\ in Doeblin and
  Fortet~\cite{doeblin} and in Johansson and \"Oberg~\cite{johob1}) or for
  counterexamples (Berger et al.~\cite{berger}), but these conditions are many
  times equivalent, owing to the compactness of $\X$ and that $g>0$. Using
  $\var_n \log g$ conforms more to the thermodynamic formalism of ergodic
  theory, since for various conditions one considers $\var_n \phi$, where the
  function $e^\phi$ is the weight function of the transfer operator. The
  choice between using $\var_n g$ or $\var_n \log g$ is usually just a matter of
  taste or context, and even in Johansson et al.~\cite{johob1} one can replace
  the uniqueness condition $\var_n \log g= o(1/\sqrt{n})$, $n\to \infty$, with the
  equivalent condition $\var_n g = o(1/\sqrt{n})$, $n\to \infty$. However, in our
  context, the two conditions are not equivalent, and we ask
  whether Theorem~\ref{thm:main} should hold with $\var_n g$ instead of
  $\var_n \log g$ and with possibly a different constant. \newline
\end{remark}

\begin{remark}
  In Johansson, \"Oberg and Pollicott~\cite{johob3} the Bernoulli propoerty
  follows when $\var_n \log g= o(1/\sqrt{n})$, $n\to \infty$. This consequence does not follow from the
  method of proof in the present paper. As was clear already from Kalikow's
  paper~\cite{kalikow}, the Bernoulli property is not automatic for a unique
  Doeblin measure. We ask under what conditions uniqueness implies the
  Bernoulli property (very weak Bernoulli property). 
  \end{remark}

\section{Proof of Theorem~\ref{thm:main}}\label{sec:uniq}
We start by constructing our main tool, namely the block-coupling.

\subsection{The block-coupling}\label{sec:dbc}

The block coupling that we use here was first developed in~\cite{johob3}, where
the authors used it to estimate $\bar d$-metrics between various chains. In this
paper we use the same block coupling, but for a different purpose, and in
particular we do not use (and thus also do not define) $\bar d$-metrics. We
start here by defining the block coupling, then, in Subsection~\ref{sec:bndtv},
we estimate its success probability, and in subsequent subsections we use it to
prove Theorem~\ref{thm:main}.

\subsubsection{The block extension $\xi^b_x$ and the block-coupling $\zeta^b_{x,\tl x}$}

Let $b\ge1$ be fixed. For a $g$-chain $x\ii n$, conditioned on the current state
$x\ii n = x$, the distribution of the state $y=x\ii {n+b}$, at $b$ time steps in
the future, is $\Ls^b \delta_x$. The transition from $x$ to $y$
means prepending a randomly sampled \emph{prefix} $p=\iota_b(y)\in A^b$ with distribution
$p \sim \xi_x^b = \Ls^b \delta_x\circ \iota_b^{-1}$. This means that we sample the prefix $p\in A^b$ with probability
\begin{equation}\label{def:xi}
  \xi_x^b(p) = \prod_{k=0}^{b-1} g(\T^k (px)).
\end{equation}
The intermediate states are then sampled with the right distribution, i.e.\  $x\ii{n+i}\sim\Ls^{i}\delta_x$, 
for $1\le i\le b$. We refer to the sampling of $y\in\{px:p\in A^b\}$ as a \emph{block extension}. In the picture of the natural
extension $\bar\x$, it amounts to exposing a \emph{block}
$\bar\x_{n+1}\bar\x_{n+2}\cdots\bar\x_{n+b}$ of $b$ coordinates of the $g$-chain
$\bar x\in\bar\X$ in one step.

Now consider a pair of $g$-chains $(x\ii n,\tl x\ii n)$ and two present states
$x=x\ii{n}$ and $\tl\x = \tl\x \ii{n}$ in $\X$. A \emph{block-coupling} of
length $b\ge1$ at $(x,\tl\x)$ is a sampling of
$(y,\tl y)=(x\ii{n+b},\tl x\ii{n+b})$ such that the marginal distributions of
$(y,\tl y)$ are block extensions. That is, a distribution
$\zeta(y,\tl y) \in\CM(\T^{-b}\x \times \T^{-b}\tl\x)$, such that
\[\text{$\zeta(y) = \Ls^b\delta_x$ and $\zeta(\tl y) = \Ls^b\delta_{\tl x}$.}\]

The event of \emph{success} means that
\[
  (y,\tl y) \in S:= \{(y,\tl y)\mid \iota_b(y)=\iota_b(\tl y)\} = {(\iota_b\times\iota_b)}^{-1}(\{(p,p) | p\in A^b\})
\]
means that the prefixes, $p=\iota_b(y)$ and $\tl p=\iota_b(\tl y)$, prepended to $x$ and
$\tl\x$ are equal.
It is a standard result that we can find a coupling
$\zeta^b_{x,\tl x}$ such that
\begin{equation}\label{dtv-prefix}
  \zeta^b_{x,{\tl\x}}(S) = 1-\dtv(\xi_x^b,\xi_{\tl x}^b),
\end{equation}
where $\dtv$ is the total variation distance. Moreover, the event $S$ of success
means that the agreement length $\kappa$ from~\eqref{defkappa} increases, i.e.\ for
$0\le i \le b$ we have
\begin{equation}\label{agreemore}
  (y,\tl y) \in S \implies \kappa(\T^i y,\T^i \tl y)= \kappa(x,\tl\x) + b-i.
\end{equation}

An important observation from~\cite{johob3} is that, for longer blocks, the
total variation distance $\dtv(\xi,\tl\xi)$ is bounded by a function of the sum of
the squares of the variations of $\log g$ rather than, say, a function of the
sum of the variations. This observation will be made precise and will be used in
Subsection~\ref{sec:bndtv} in order utilize our variational condition for the
purpose of controlling the coupling.

\subsubsection{The associated Markov chain}
The way in which we use the block coupling is to construct a pair of coupled
$g$-chains $(x\ii{n},\tl\x\ii{n})$, $n\ge0$, that agree in the coordinates to a
large extent. Similarly to~\cite{johob}, we associate the pair
$(x\ii{n},\tl\x\ii{n})$ of coupled $g$-chains to a Markov chain $Y\ii{n}$ with
state space $\ZZ$ such that for all $m\ge0$, we maintain the bound
\begin{equation}\label{Ydom}
  \kappa(x\ii{m},\tl\x\ii{m})\ge Y\ii{m}.
\end{equation}

We now explain the transition rules of the Markov chain $\big(Y^{(n)}\big)$, and
then we illustrate its coupling with our Doeblin chains. Let
$\CB=\{B_\ell:\ell\ge0\} \subseteq \NN_0$ be a set of states, which we call ``breaking states'',
ordered so that $0=B_0<B_1<B_2<\dots$. If the Markov chain is in a state outside
of $\CB$, then it gets incremented by $1$ with probability $1$. If the Markov
chain is in a state in $\CB$, then the transition is more complex, see precise
description below. Hence the name `breaking states'. Let $b_\ell=B_{\ell+1}-B_{\ell}$ be
the distance between two consecutive breaks. Let also
\begin{equation}\label{y-prob}
  p_\ell = \min\big\{1/2,1 - \inf_{x,\tl x} \{ \dtv(\xi^{b_\ell}_{x},\xi^{b_\ell}_{\tl x}) \mid \kappa(x,\tl\x) \ge B_\ell \}\big\}.
\end{equation}
We start the chain at $Y\ii0=0$ and, for $n\ge0$, we let
\begin{equation}\label{Ydef}
  Y\ii{n+1} = \begin{cases}
                Y\ii n  +1 & \text{w. prob. 1 if } Y\ii n\not\in\CB \\
                Y\ii n  +1 & \text{w. prob. $p_\ell$ if } Y\ii n=B_\ell\in\CB \\
                -b_\ell + 1 & \text{w. prob. $1-p_\ell$ if } Y\ii n=B_\ell\in\CB. \\
              \end{cases}
\end{equation}
Thus, the chain $Y\ii{n}$ moves deterministically upwards at unit speed, i.e.
$Y\ii{n+1}=Y\ii n +1$, whenever $Y\ii n\not\in\CB$. In case it reaches a breaking
state $Y\ii{n} = B_\ell$, it continues upwards with probability $p_\ell$ and with
probability $1-p_\ell$, it jumps to the state $-b_\ell+1$ below zero.

If $Y\ii n = B_\ell$ then~\eqref{Ydom} implies $\kappa(x\ii n,\tl x\ii n)\ge B_\ell$. Hence, the definition~\eqref{y-prob} implies that
\[ p_\ell \le \zeta^{b_\ell}_{x\ii n,\tl x\ii n}(S).\] Hence, we may couple the event
$Y\ii{n+1}=Y\ii n + 1$ in state $B_\ell$ with a \emph{successful} block extension
$(x\ii{n+b_\ell},\tl x\ii {n+b_\ell})$ of the chains. In this case, the chain $Y\ii m$
moves deterministically upwards until it reaches the next state
$B_{\ell+1}=B_\ell + b_\ell$ at time $n+b_\ell$ and relation~\eqref{agreemore} implies
that~\eqref{Ydom} remains valid for times $m$ where $n\le m \le n+ b_\ell$. In case
$Y\ii {n+1}=-b_\ell+1$, the chain moves deterministically upwards until it hits
state zero at time $n+b_\ell$. In this case~\eqref{Ydom} remains trivially true for
$n\le m \le n+b_\ell$.

Lemma~\ref{lem:dtv} below shows that there exist $K$ and $N$ (both large), such that the definition $B_0=0$ and for $\ell=1,2,\ldots$
\begin{equation}\label{B_elldef}
  B_\ell := N \cdot K^{\ell-1}, \quad\text{i.e.\ $b_\ell=(K-1)B_\ell$.}
\end{equation}
implies
\begin{equation}\label{p_ellbnd}
p_\ell  > \frac1K.
\end{equation}
 Thus, we define the breaking states $\CB$ according
to~\eqref{B_elldef} above.

In Subsection~\ref{sec:bndtv} we state and prove Lemma~\ref{lem:dtv} which we referred to before.
In Subsection~\ref{sec:analmark}, we use renewal
theory and a theorem by Kesten~\cite{kesten} to further study the chain
$Y\ii{n}$. Then in Subsection~\ref{sec:pfmn} we use this construction to prove
Theorem~\ref{thm:main}.

\subsection{Bounding the total variation distance}\label{sec:bndtv}
In order to provide the constant 2 in our theorem, we use some familiar
estimates of the total variation metric based on the Hellinger integral
(cf.~\cite{johob3},~\cite{jacod}). More precisely,  let $\xi=\xi_x^b$ and $\tl\xi=\xi_{\tl x}^b$ be as in~\eqref{dtv-prefix} so that
  $\zeta_{x,\tl\x}^b(S) = 1-\dtv(\xi,\tl\xi)$. The Hellinger integral is useful for us
  because of its use to bound the total variation distance. Indeed, using
  Cauchy-Schwarz, it is easy to verify (as in e.g.~\cite{jacod}) that
  \begin{equation}\label{eq:std}
    1-\dtv(\xi,\tl\xi) \ge 1-\sqrt{1-H^2} \ge \frac {H^2}2,
  \end{equation}
 where, for discrete measures, the Hellinger integral has the form
  \begin{equation*}
    H := H(\xi,\tl \xi) := \sum_{p\in A^b} \sqrt{\xi(p)\,\tl\xi(p)}.
  \end{equation*}

\begin{lemma}\label{lem:dtv}
  Assume~\eqref{eq:assumption} holds. Let $N\in\NN$. For a large enough integer $K\ge2$ the minimum success probability satisfies
  \begin{equation}
    \inf_{B \in \NN ;\x,\tl\x \in\X}
    \left\{ \zeta_{x,\tl\x}^b(S) : \kappa(x,\tl\x) \ge B \ge N\right\}
    > \frac 1K
    \label{eq:dtv}
  \end{equation}
  where $b = b(B) = (K-1)\cdot B$.
\end{lemma}

\begin{proof}

  Let $K\ge 2$ and $B\ge N\ge 1$ be integer numbers and define $b = (K-1)\cdot B$.  
  For $p\in A^b$, let $y=px$ and $\tl y=p\tl\x$. Using
  $\xi(p) = \prod_{j=0}^{b-1} g(\T^j y)$ and $\tl\xi(p) = \prod_{j=0}^{b-1} g(\T^j\tl y)$,
  we obtain
  \begin{equation}
    H  = \sum_{p \in A^b} \prod_{j=0}^{b-1} \sqrt{g(\T^j y)g(\T^j\tl y)}.
  \end{equation}
  We have for all $j\in\{0,\ldots, b-1\}$
  \begin{equation}
    \sqrt{g(\T^j y)\,g(\T^j\tl y)} =
    \frac{g(\T^j y) + g(\T^j\tl y)}{2} \cdot \frac1{\cosh(\delta_j)},
  \end{equation}
  where
  \[
    \delta_j(u) = \left| \frac{\log g(\T^j y) - \log g(\T^j\tl y)}{2} \right|
    \le \frac {r_{B+b-j}}{2},
  \]
  since, by \eqref{agreemore}, $\kappa(\T^j y,\T^j \tl y)\ge B+b-j$ 
  when $\kappa(x,\tl\x)\ge B$. Since, for
  all $w,\tl w\in\X$, we have \(\sum_{a\in A} (g(aw)+g(a\tl w))/2 = 1\), it follows
  that
  \[
    \sum_{p \in A^b} \prod_{j=0}^{b-1} \frac{g(\T^j y) + g(\T^j\tl y)}{2} = 1.
  \]
 Since $B+b=KB$, we have
  \[
    H \ge \exp\left(-\sum_{k=B+1}^{KB} \log\cosh(r_k/2)\right).
  \]

Since $\log\cosh (u/2) = \frac 18 u^2 + \Ordo{u^4}$, 
we have for $c = \limsup r_n\sqrt{n}$ and a constant $c_1\ge0$ that
  \begin{align*}
    \log H &\ge -\frac 18\,  
    \sum_{k=B+1}^{KB} r^2_{k} -
    c_1\sum_{k=B+1}^{KB} r^4_{k} \ge
    -\frac{c^2}8 \cdot \sum_{k=B+1}^{KB} \frac 1k -
    c^4 \cdot c_1 \sum_{k=B+1}^{KB} \frac 1{k^2}.
  \end{align*}
  Using the formula
  \[
    \sum_{k=1}^n \frac 1k = \log n + \gamma +\frac {1}{2n} + \Ordo{\frac 1{n^2}}
  \]
  for the partial sums of the harmonic series and the integral estimate 
  \[
  \sum_{k=B+1}^{KB} \frac 1{k^2} \le \frac 1{B+1} = \Ordo{1/B},
  \]
  we obtain
  \(
    \log H  \ge -\frac {c^2}8 \log K - {\frac {c_2}B},
  \)
  for some constant $c_2\ge0$.
  Thus
  \begin{equation}\label{H2}
   H^2 \ge \frac 2K \cdot \exp\left({\epsilon\log K -\log 2 - \frac {2c_2}B}\right),
  \end{equation}
  where $\epsilon:= 1 - c^2/4$, which is positive by assumption \eqref{eq:assumption}. 
  We can choose $K\ge2$ large enough so that 
  ${\epsilon\log K - \log 2 - \frac {2c_2}B} > \epsilon/2$. \label{eq:logH}
  Using~\eqref{eq:std} with the bound \eqref{H2} we obtain
  \[
    1-\dtv \ge \frac 1K \cdot e^{\epsilon/2}> \frac 1K,
  \] 
  which concludes the proof of Lemma~\ref{lem:dtv}.
\end{proof}


\subsection{A property of the associated Markov chain $Y\ii n$}\label{sec:analmark}

\begin{proposition}\label{prop:limsupY}
  Under the assumptions of Lemma \ref{lem:dtv}
  \begin{equation}\label{eq:limsupY}
    \limsup_{n\to\infty} \frac{Y\ii{n}}{n} = 1  \; a.s.
  \end{equation}
\end{proposition}

\begin{remark}
We do not have a similar control for the $\liminf$, and in fact one can produce examples where the $\liminf$ is strictly smaller than $1$.
\end{remark}

\begin{proof}[Proof of Proposition \ref{prop:limsupY}]
  In order to establish \qr{eq:limsupY}, we let $T_k$, $k\ge0$, denote the renewal
  process of return times to state $0$ of the Markov chain $Y\ii{n}$, i.e.\
  $T_0=0$ and $T_k = \inf \{n>T_{k-1} : Y_n = 0\}$. The differences
  $M_k=T_k-T_{k-1}$, $k=1,2,\dots$, constitute an iid sequence. Note that,
  $M_k = B_{L_k+1}$ where $B_{L_k}$ is the last visited breaking state by
  $Y\ii{n}$ in excursion number $k$, i.e.\
  $L_k:=\sup\{ \ell\in\ZZ_{\geq 0}: Y\ii{n} =B_\ell, n\in \interval[open right]{T_{k-1}}{T_k}\}$. If $L_k=0$ then
  $M_k=N$ and, otherwise, if $L_k\ge1$, we have
  \begin{equation}\label{eq:supYrelM}
    \sup \{Y\ii n : T_{k-1} < n \le T_k\} = B_{L_k} = M_k/K.
  \end{equation}
  Since, a.s., $L_k\ge1$ infinitely often, we have, almost surely,
  \begin{align*}\label{eq:limsupp}
    \limsup_{n\to\infty}
    \frac{Y\ii{n}}{n} &\ge \limsup_k \frac{M_k/K}{T_{k-1} + M_k/K}\\
                      & \ge 1 - \liminf \frac {K \cdot T_{k-1}}{M_k}.
  \end{align*}
  which equals $1$ almost surely if and only if
  \begin{equation}
    \label{eq:kesten}
    \limsup \frac{M_k}{T_{k-1}}
    = \limsup \frac{M_k}{M_1+M_2+\cdots+M_{k-1}} = \infty, \quad\text{a.s.}
  \end{equation}
  A result by Kesten (\cite{kesten}, 1971) says that \qr{eq:kesten} holds as
  soon as
  \begin{equation}\label{eq:infE}
    \E\left(M_k\right) = \infty,
  \end{equation}
  which is true in our case. We have
  \[
    \E( M_k ) = \E( K^{L_k}N ) = \sum_{\ell=0}^{\infty} \P(L_k=\ell) K^{\ell}N = \infty,
  \]
  since, for $\ell\ge1$,
  \[
    \P(L_k=\ell) =
    \left(\prod_{i=0}^{\ell-1} p_i\right) \cdot (1-p_\ell)
    > p_0 {\Big(\frac 1K\Big)}^{\ell -1} \cdot \frac 12,
  \]
  where last inequality follow from Lemma~\ref{lem:dtv} and
  $p_\ell\le \frac 12$ by \eqref{y-prob}. This establishes that the construction
  satisfies \qr{eq:limsupY}.
\end{proof}

\subsection{Finishing the proof of Theorem~\ref{thm:main}} \label{sec:pfmn}

\begin{proof}[Proof of Theorem~\ref{thm:main}]
  Assume for contradiction that there are two distinct Doeblin measures $\mu$ and
  $\tl\mu$ with respect to $g$. Since the space $\CM_g$ of Doeblin measures with
  respect to $g$ is convex and compact,
  we may assume that $\mu$ and $ \tl\mu$ are
  extremal, i.e.~\emph{ergodic}.

  Now, we choose the initial states $x\ii 0$ resp. $\tl\x\ii 0$ independently
  according to the distributions $\mu$ resp. $\tl\mu$, and then complete them to
  chains $x\ii{n}, \tl\x\ii{n}\in \X$, $n\ge1$, according to the coupling from
  Subsection~\ref{sec:dbc}. Since we sample the initial states using stationary
  distributions, we can assume that $x\ii{n}$ and $\tl\x\ii{n}$ have
  distribution $\mu$ and $\tl\mu$, for all $n\in\ZZ$.

  Since $\mu\not=\tl\mu$, we have some local function $f$ such that
  $\int f \, d \mu>\int f\, d\tl\mu$. Let $N$ be s.t. $f$ is in
  $\CF_N$.
  Then,
  by the pointwise Birkhoff Ergodic Theorem, almost surely, the following limits
  exist
  \[
    \lim_{n\to\infty} \frac 1n \sum_{i=1}^n f(\x\ii i) = \int f\,d\mu
    > \int f\,d\tl\mu = \lim_{n\to\infty} \frac 1n \sum_{i=1}^n f(\tl\x\ii i).
  \]
  In particular, it must hold almost surely that
  \begin{equation}\label{eq:mstdif}
    \liminf_{n\to\infty} \frac 1n \sum_{i=1}^n \big(f(\x\ii i) - f(\tl\x\ii{i})\big) > 0.
  \end{equation}
  The fact that $f$ is $\CF_N$-measurable means that
  $f(x\ii i)-f(\tl\x\ii i) = 0$ if $\kappa(x\ii i,\tl\x\ii i) \ge N$. Thus, by
  \qr{Ydom}, we have
  \[
    \frac 1n \sum_{i=1}^n \left(f(\x\ii i) - f(\tl\x\ii{i})\right)
    < \frac 1n \sum_{i=1}^n 2\| f \|_\infty \cdot \ett_{Y\ii i \le N}
    < 2 \| f \|_\infty \frac{n-Y\ii n +N}{n}.
  \]
  However,~\eqref{eq:limsupY} implies that the right hand side above is
  arbitrarily close to zero for infinitely many $n$, in contradiction
  to~\eqref{eq:mstdif}. We conclude that there can only be one Doeblin measure
  corresponding to $g$.

\end{proof}

\begin{remark}
  Here we use ergodicity in the sense of stochastic processes, i.e.\ that
  \[
    \lim \frac 1n \, \sum_{k=1}^n f(x\ii k) = \int f \,d\mu \quad\text{$\mu$-a.e.},
  \]
  i.e.\ in terms of iterations over ``future'' values. In ergodic theory,
  Birkhoff's theorem is usually stated as
  \[
    \lim \frac 1n \, \sum_{k=1}^n f(\T^k\x) = \int f \,d\mu \quad\text{$\mu$-a.e.},
  \]
  i.e.\ in terms of iterations over ``past'' values. But, for Doeblin measures,
  we have equivalence between these two statements, since
  \[
    (\X,\T,\mu)\ \text{ergodic} \iff
    (\bar\X,\bar\T,\bar\mu)\ \text{ergodic} \iff
    (\bar\X,\bar\T^{-1},\bar\mu)\ \text{ergodic},
  \]
  where $(\bar\X,\bar\T,\bar\mu)$, $\pi: \bar\X\to\X$, is the natural extension of
  $(\X,\T,\mu)$. But then,
  \[
    \lim \frac 1n \, \sum_{k=1}^n f(x\ii k) =
    \lim \frac 1n \, \sum_{k=1}^n f(\pi(\bar\T^{-k}\bar\x))
    \underset{\text{a.e.}}{=}
    \int f(\pi(\bar\x)) d\,\bar \mu(\bar\x) = \int f\,d\mu.
  \]
\end{remark}

\section{A non-mixing unique Doeblin measure}\label{sec:nonu}

\def\I{\mathsf{I}} In this section, we answer a well-known folklore question
that we in particular attribute to Peter Walters, namely whether uniqueness of a
Doeblin measure (in the case of finitely many symbols and $g>0$) implies that the iterates of
the transfer operator $\L_g$ converge. In fact we show by means of an example
that this is not necessary and that the unique Doeblin measure in our
construction is not even weakly mixing.

Recall that the transfer operator here is defined for continuous functions $f$
by
\[\L_g f(x)=\sum_{y\in T^{-1}x} g(y)f(y).\]

We know from Breiman~\cite{breiman} that a unique Doeblin measure $\mu$ implies
that for any continuous function $f:X\to\RR$, we have
\begin{equation}\label{breimp}
  \lim_{n\to \infty} \sup_x \bigg|\frac{1}{n} \sum_{k=1}^n \L_g^k f(x)-\int f\; d\mu\bigg|=0.
\end{equation}
The full uniform convergence of the iterates of the transfer operator implies
{\em strong mixing} for any Doeblin measure $\mu$. That is if
\[\lim_{n\to \infty} \sup_x \bigg| \L^n f(x)-\int f\; d\mu\bigg|=0,\]
then we also have the strong mixing of $\mu$:
\[\mu(\T^{-n} A \cap B) \to \mu(A)\mu(B),\]
as $n\to \infty$, for any Borel subsets $A$ and $B$ of $X$.

We construct a unique Doeblin measure $\mu$ that is not even {\em weakly mixing},
hence not strongly mixing, and then the iterates of the transfer operator do
not necessarily converge (uniformly or pointwise). Weak mixing for a measure $\mu$
means that
\[\frac{1}{n} \sum_{k=1}^n |\mu(\T^{-n} A \cap B)- \mu(A)\mu(B)|=0,\]
for any Borel subsets $A$ and $B$ of $X$. Observe that this does not follow from
Breiman's property~\eqref{breimp}.

\subsection{The construction}

Let $\CK_g$ denote the set of $g$-chains (of infinite age) on the state space $\X$.
The members of $\CK_g$ have also been studied under the name $G$-measures, see
e.g.~\cite{dooley1, dooley2}. The $g$-chain condition takes the form
\begin{equation}\label{gchain}
  \mu( x \mid \T^n x ) = g(x) \cdot g(T x) \cdots g(T^{n-1}x) \quad\text{for all $n\ge0$}.
\end{equation}
The space of Doeblin measures $\CM_g(\X)$ is the set of \emph{stationary}
$g$-chains. Our aim is to demonstrate that there is a system $(\X,\T,\tilde g)$
where the Doeblin measure is unique but there are multiple $g$-chains: That is,
\begin{equation}\label{goal}
  \text{$\CM_g=\{\mu\}$ and $\CM_g \subsetneq \CK_g$.}
\end{equation}

Our aim is to show the following theorem.
\begin{theorem}\label{thm:counter}
  There exists a continuous Doeblin function $g>0$ on $\X = A^\NN$, $|A|<\infty$, for
  which \qr{goal} holds.
\end{theorem}
Since $\T$ is the full shift on $X$, we know from Walters~\cite{walters2}, p.\ 334, p.\ 340)
 that we have a unique measure in $\CK_g$ if and only if we have uniform (or
  equivalently: pointwise) convergence of the sequence $\{\L^n f(x)\}$. Hence we have the following corollary.
\begin{corollary}\label{thm:noptwise}
  There exists a Doeblin function $g$ with a unique Doeblin measure such that
  we do not have pointwise convergence of the sequence of iterates of the associated transfer operator $\L_g$.
\end{corollary}
Our result also implies that there exists a system with a unique Doeblin measure
which is not Bernoulli, but that was already proved by Kalikow~\cite{kalikow}.

\subsection{The proof of Theorem~\ref{thm:counter}}

Recall that an involution on $\X$ is a bijection $\F:\X\to\X$ such that
$\F\not=\I$ and $\F^2=\I$, where $\I$ is the identity map. If a Doeblin function
$g(x)$ is invariant under an involution $\M$ commuting with $\T$ then both
$\CM_g$ and $\CK_g$ are invariant under the direct image $\M_*:\CM(\X)\to\CM(\X)$,
$\nu\mapsto \nu\circ\M^{-1}$. We look for a triple $(\X,\T,g)$ with two distinct ergodic
measures $\nu_+$ and $\nu_-$ conjugated by an involution symmetry $\M$. That is, we
have
\begin{equation}\label{gM}
  g(\M x) = g(\x), \quad \M\T=\T\M, \quad\M^2=\I\quad\text{and}\quad \nu_-=\nu_+\circ \M^{-1}.
\end{equation}
We also assume that there is a second involution $\F:\X\to\X$, $\F\not=\M$, with
the commutation rules
\begin{equation}\label{FTcomm}
  \M\F = \F\M \quad \text{and}\quad\F\T=\M\T\F.
\end{equation}
In addition, we also assume that $\nu_+$ and $\nu_-$ are the \emph{only} extremal
$g$-chains, so that $\CM_g$ and $\CK_g$ are the convex hull
\begin{equation}\label{sole}
  \CK_g=\CM_g=\operatorname{conv}\{ \nu_+,\nu_-\}.
\end{equation}

A candidate for a system that satisfies~\eqref{gM},~\eqref{FTcomm}
and~\eqref{sole}, would be the Doeblin function constructed by Bramson and
Kalikow~\cite{bramson} on $\X={\{+1,-1\}}^\NN$ with $\T$ the usual shift. With
$\M[(x_i)]=(-x_i)$ and $\F[(x_i)] = \left({(-1)}^i x_i\right)$ it is easy to
verify~\eqref{gM} and~\eqref{FTcomm}. However, the condition~\eqref{sole} seems
hard to prove. Instead, we use below a system on four symbols constructed
in~\cite{berger}, originally for the purpose of providing a system with
multiple Doeblin measures, but with $\ell^{2+\epsilon}$-summable variations of $g$.

\begin{lemma}\label{lem:commute}
  Assume a system $(\X,\T,g)$ satisfies~\eqref{gM},~\eqref{FTcomm}
  and~\eqref{sole}. Let $\tl g(x) = g(\F x)$ and $\tl\nu_{\pm} = \nu_{\pm}\circ\F^{-1}$.
  Then, for the system $(\X,\T,\tl g)$, we have
  \[
    \CM_{\tilde g} = \{ \frac12(\tl\nu_+ + \tl\nu_-)\} \quad \text{ and } \quad \CK_{\tl g} = \operatorname{conv}\{\tl\nu_+,\tl\nu_-\}.
  \]
\end{lemma}

\begin{proof}
  We claim that the map $\nu \mapsto \tl\nu = \nu\circ\F^{-1}$ takes $\CK_g$ bijectively onto
  $\CK_{\tl g}$. To see this, let $\mu\in\CK_g$ be a $g$-chain and note that
  \[
    \tl \mu (\x\mid \T^k \x) = \mu(\F x | \T^{k} \F x)
  \]
  Then it follows that
  \begin{align*}
    \tl\mu(x\mid \T^n x) &= g(\F x) \cdot g(\T\F x) \cdot g(\T^2 \F x) \cdots g(\T^{n-1}\F x) \\
                    &= g(\F x) \cdot g(\M\F\T x) \cdot g(\M^2\F\T^2 x) \cdots g(\M^{n-1}\F\T^{n-1} x) \\
                    &= \tl g(x) \cdot \tl g(\T x) \cdot \tl g(\T^2 x) \cdots \tl g(\T^{n-1} x).
  \end{align*}
  Since this holds for all $k\ge0$, it shows that $\tl \mu$ is a $\tl g$-chain. The
  first equality is due to $\mu$ being a $g$-chain. The second equality follows
  from the rule~\eqref{FTcomm}. The final equality is due to $g(\M^k x)=g(x)$ by
  the symmetry and that $g(\F x)=\tl g(x)$ by definition. Since $\F$ is an
  involution, it is clear that the argument gives a bijection between $\CK_g$ and
  $\CK_{\tl g}$.

  None of the measures $\tl\nu_+$ and $\tl\nu_-$ are translation invariant since
  \[
    \tl\nu_+ \circ \T^{-1} = \nu_+\circ\F\circ\T^{-1} = \nu_{+} \circ \M \circ \F = \nu_-\circ\F = \tl \nu_-.
  \]
  Thus chains with distribution $\tl\nu_{\pm}$ are periodic in distribution with
  period 2. It is then clear that the only candidate for a stationary
  $\tl g$-chain in $\CK_{\tl g} = \operatorname{conv} \{\tl\nu_+,\tl\nu_-\}$ is
  the midpoint $\tl\mu=(\tl\nu_++\tl\nu_-)/2$ between them.
\end{proof}

To prove Theorem~\ref{thm:counter}, we use the following counterexample to
uniqueness of a Doeblin measure based on the example provided in~\cite{berger}.
We consider an alphabet $A =\{0,1\} \times \{-1,1\}$ and the input to the
Doeblin function $g$ as two sequences: $(x_{n})_{n\geq 0}$ with values in
$\{-1,1\}$ and $(y_{n})_{n\geq 0}$ with values in $\{0,1\}$.

We construct a Doeblin function $g$ as a mechanism for sampling $x_0$ and $y_0$
given $(x_{n})_{n\ge1}$ and $(y_{n})_{n\ge1}$, using the following scheme:
\begin{enumerate}[(1)]
  \item Sample $y_0$ as a $1/2$-Bernoulli variable independently of
        $(x_{n},y_n)_{n\ge1}$.

  \item Using $(y_{n})_{n\ge1}$ choose a finite set $S \subseteq \{k: k\ge1\}$.
        Either $|S|$ is odd or $S=\emptyset$.

  \item Using $(y_{n})_{n\ge1}$ choose a number $\xi$ between $.5$ and $.75$.

  \item If $S\neq \emptyset$, then take $x_0=\sgn \sum_{k\in S} x_k$ with probability $\xi$ and
        $x_0=-\sgn \sum_{k\in S} x_k$ with probability $1-\xi$. If $S=\emptyset$, choose
        $x_0=1$ with probability $1/2$ and $x_0=-1$ with probability $1/2$.
\end{enumerate}

The following theorem is the main result of~\cite{berger} in combination with the appendix of the present paper, where we show, for the sake of clarification, that there are precisely two ergodic measures in this example.
\begin{theorem}[\cite{berger}]\label{thm:bhs}
  There exists a way to choose the set $S$ and the parameter $\xi$ such that the
  Doeblin function $g$ has exactly two ergodic (extremal) Doeblin measures $\mu_+$
  and $\mu_-$.

  Moreover, $\mu_+$ and $\mu_-$ are also the only extremal Doeblin chains for $g$.
\end{theorem}

Now, we define the action of $\M$ and $\F$ as in the Bramson-Kalikow example
mentioned above, but restricted to the $x$-sequence. That is, let
$\M((x_i,y_i)_{i=1}^\infty) = (-x_i,y_i)_{i=0}^\infty$ and
$\F((x_i,y_i)_{i=1}^\infty) = ((-1)^i x_i, y_i)_{i=1}^\infty$. It is then straightforward
to verify that the commutation rules~\eqref{gM} and~\eqref{FTcomm} hold. From
Theorem~\ref{thm:bhs} we also demonstrate~\eqref{sole}. An application of
Lemma~\ref{lem:commute} thus concludes the proof. \qed

\begin{remark}
  It is still an open question whether the two measures constructed by Bramson
  and Kalikow~\cite{bramson} are the only extremal measures in $K_g$. In the
  affirmative case, their construction can also be adapted to provide a similar
  counterexample.
\end{remark}

\subsection{Proof of Theorem~\ref{thm:main2}}

We conclude by giving the proof of Theorem~\ref{thm:main2} since it uses the
construction in the proof of Theorem~\ref{thm:counter}. We claim that the
translation invariant measure $\tl \mu=(\tl \nu_+ + \tl \nu_-)/2$ constructed above is
not weakly mixing. Let $A$ be a set on which $\tl\nu_+ $ is concentrated and such
that $\tl \nu_-(A)=0$ ($\tl \nu_+$ and $\tl \nu_-$ are mutually singular). We obtain
an oscillating sequence
\[
  \mu((\T^{-n} A) \cap A) = \nu_+(\T^{-n} A) =
  \begin{cases}
     1/2 & \text{$n$ even} \\
     0 & \text{$n$ odd.}
   \end{cases}
\]
Thus the terms in the Ces\`aro mean
\[\frac{1}{n} \sum_{k=1}^n |\mu(\T^{-n} A \cap A)- \mu(A)\mu(A)| \]
are all equal to $1/4$ and hence the limit is $1/4$ and not zero, disproving
weak mixing.

\begin{remark}
  Note that Theorem~\ref{thm:main2} does not directly follow from
  Theorem~\ref{thm:counter} and Corollary~\ref{thm:noptwise}, since strong
  mixing implies weak mixing. The existence of a non weakly mixing ergodic
  measure will hold whenever the map $\nu\to\nu\circ\T^{-1}$ have periodic points in
  $\CK_g\setminus\CM_g$.
\end{remark}

\begin{remark}
  In our example, the Doeblin measure is ergodic, but not totally ergodic. We
  may ask whether a unique totally ergodic Doeblin measure has to be weakly
  mixing.
\end{remark}

\appendix
\newcommand{\sign}{\operatorname{sign}}
\section{Exactly two extremal measures in Theorem \ref{thm:bhs} }
In this appendix we prove a fact that should have been proven in~\cite{berger}.
Both the fact and its proof were known to the authors of~\cite{berger} when they
wrote their paper. All throughout this appendix we use the terminology and the
notation of~\cite{berger}, and we assume knowledge of that paper.

Let $A=\{-1,+1\}^2$, and let $g$ be the Doeblin function ('specification' in the
terminology of~\cite{berger}) defined in~\cite{berger}. In order to be
consistent with the notations in~\cite{berger}, we think of $g$ as a function of
$ (x_{-j}, y_{-j} : j\in \mathbb N_0). $

The goal of this appendix is to prove that indeed there are exactly two Doeblin
measures. In~\cite{berger} it was proven that there are at least two, so here we
only need to prove that there are no more than two.

\begin{proof}[Proof that there are exactly two extremal Doeblin measures in
  Theorem \ref{thm:bhs}]
  Let $K$ be as in section 2 of~\cite{berger}, and for every $k \geq K$, let $O_k$
  be the beginning and let $C_k$ be the opening of the $k$ block containing $0$
  (see Definitions 3 and 4 in~\cite{berger}, p.\ 1345 in the journal version).
  Let
  \[
    X_k = \sign\sum_{t\in C_k} x_t
  \]
  be the signaure of the block. By (the proof of) Lemma 7 of~\cite{berger} (p.\
  1350), if $\mu$ is a Doeblin chain for $g$, then
  \[
    X = \lim_{k \to \infty} X_k
  \]
  exists $\mu$-almost surely. Further, if $\mu$ is extremal, then the limit is an
  almost sure constant. By the definition of $X(k)$ and the fact that the
  openings are of odd sizes, we get that $X$ has to be $+1$ or $-1$.

  The proof will be complete once we have shown that if $X$ has the same
  distribution under two Doeblin measures $\mu_1$ and $\mu_2$ then $\mu_1 = \mu_2$.
  Without loss of generality we assume that
  \begin{equation}\label{eq:mumu}
    \mu_1(X=1) = \mu_2(X=1) = 1.
  \end{equation}

  In order to prove~\eqref{eq:mumu}, for every $\epsilon$ and $N$ we find a coupling
  $P$ between $\mu_1$ and $\mu_2$ such that
  \[
    P\big(\forall_{-N < j < N} \; x^{(1)}_j = x^{(2)}_j \mbox{ and
    } y^{(1)}_j = y^{(2)}_j \big) > 1 - \epsilon.
  \]

  Since the sequence $(y_j)$ has the same marginal distribution in $\mu_1$ and in
  $\mu_2$ (and in fact in every Doeblin chain of $g$), we first sample
  $\big(y^{(1)}_j\big)_{j \in \mathbb Z}$ and
  $\big(y^{(2)}_j\big)_{j \in \mathbb Z}$ in a way that $y^{(1)}_j = y^{(2)}_j$
  for all $j$. In particular, this implies that both sequences have exactly the
  same block structure.

  For $k \geq K$ let
  \[
    A(k) = A_N(k) = \{O_k \neq \emptyset \} \cap \{O_k \cap [-N,N] \neq\emptyset\}.
  \]

  Now let $k_0$ be so large that
  \begin{equation}\label{eq:pgood}
    \mu_1 \big( X_{k_0} \neq 1 \big) + \mu_2 \big( X_{k_0} \neq 1 \big) + \mu \big( A_N(k_0)^c \big) < \epsilon.
  \end{equation}
  where for the events related to $(y_j)$ we simply write $\mu$ because, as
  already mentioned, all Doeblin measures agree on the marginal distribution of
  the sequence $(y_j)$.

  Condition on the event $A_N(k_0)$. On the complement event our coupling fails,
  but the probability of $A_N(k_0)^c$ is small enough.

  Let $j_0 = \max(O_{k_0})$. We now sample $\big( x^{(1)}_{j} \big)_{j < j_0}$
  according to $\mu_1$ and $\big( x^{(2)}_{j} \big)_{j < j_0}$ according to $\mu_2$, say independently.
  Note that
  $ P\left(A(k_0)^c \cup \big\{X^{(1)}_{k_0}\neq X^{(1)}_{k_0}\big\} \right) < \epsilon, $ so
  in what follows we condition on the complement event
  $ A(k_0) \cap \big\{X^{(1)}_{k_0} = X^{(1)}_{k_0}\big\}. $ We now continue to
  sample $\big( x^{(1)}_j, x^{(2)}_j \big)_{j= j_0 + 1}^\infty$ sequentially as
  follows: Once we have $\big( x^{(1)}_i, x^{(2)}_i \big)_{i < j}$, we sample
  $x^{(1)}_j$ and $x^{(2)}_j$ according to the probabilities given by $g$,
  coupled in a way that maximizes $P\big( x^{(1)}_j = x^{(2)}_j \big)$.

  Clearly $P$ is a coupling of $\mu_1$ and $\mu_2$, and we need the following fact.


  \begin{lemma}\label{lem:inint}
    \[
      P\left[\left. \forall_{j \in [-N,N]} \ x^{(1)}_j = x^{(2)}_j \ \right| \ A(k_0) \cap \big\{X^{(1)}_{k_0} = X^{(1)}_{k_0}\big\} \right] = 1.
    \]
  \end{lemma}

\begin{proof}
  For every $j$ write $S_j$ for the set whose majority we take in deciding
  $x^{(1)}_j$ and $x^{(2)}_j$. By induction, for every $j > j_0$, we have that
  either $S_j = C_{k_0}$ or $S_j \subseteq [j_0 + 1, \infty] \cap \mathbb Z$. Therefore, again
  by induction, on the event
  $A(k_0) \cap \big\{X^{(1)}_{k_0} = X^{(1)}_{k_0}\big\}$ we have that
  $x^{(1)}_j = x^{(2)}_j$ for all $j > j_0$.
\end{proof}

Using Lemma \ref{lem:inint} we see that the total variation between the
restrictions of $\mu_1$ and $\mu_2$ on $[-N,N]$ is arbitrarily small. As $N$ is
arbitrarily large, we get that $\mu_1 = \mu_2$.

\end{proof}

\end{document}